\documentclass[11pt,a4paper]{article}
\usepackage[T1]{fontenc}
\usepackage{lmodern}
\usepackage{amsmath,amssymb,amsthm,mathtools}
\usepackage[margin=25mm]{geometry}
\usepackage{booktabs,microtype,cite,setspace,fancyhdr}
\newif\ifreviewspacing
\reviewspacingfalse
\ifreviewspacing\fi
\usepackage[hidelinks]{hyperref}
\hypersetup{pdftitle={Coefficient arithmetic and Mellin moments of the Rogers--Ramanujan continued fraction},pdfauthor={K. Srinivasa Raghava},pdfsubject={Signs in arithmetic progressions, coefficient densities, and Mellin moments},pdfkeywords={Rogers-Ramanujan continued fraction, coefficient arithmetic, sign density, arithmetic progressions, Mellin transform}}
\newtheorem{theorem}{Theorem}[section]
\newtheorem{lemma}[theorem]{Lemma}
\newtheorem{proposition}[theorem]{Proposition}
\newtheorem{corollary}[theorem]{Corollary}
\theoremstyle{remark}

\newcommand{\QQ}{\mathbb Q}
\newcommand{\ZZ}{\mathbb Z}
\newcommand{\CC}{\mathbb C}
\DeclareMathOperator{\Res}{Res}
\numberwithin{equation}{section}
\title{Coefficient arithmetic and Mellin moments of the Rogers--Ramanujan continued fraction}
\author{K. Srinivasa Raghava\thanks{Pie Mathematics Association, Chennai, India. Email: \href{mailto:srinivasaraghavak@gmail.com}{srinivasaraghavak@gmail.com}.}}
\date{}
\begin{document}
\maketitle
\thispagestyle{fancy}
\begin{abstract}
We study the signs of logarithmic coefficients associated with the
Rogers--Ramanujan continued fraction. An explicit criterion involving a
finite product determines whether all coefficients in an arithmetic progression are
negative or both signs occur with positive natural density. On the sets
$n=5^a m$, $5\nmid m$, negative coefficients occur with positive
relative density $\delta_a^-$ at every fixed $a$. If $m_a$ is the least
such $m$, we prove that both $5^{-a}\log\log(1/\delta_a^-)$ and
$5^{-a}\log\log m_a$ tend to $5\log\phi/(12e^\gamma)$, where
$\phi=(1+\sqrt5)/2$ and $\gamma$ is Euler's constant. These are leading
double-logarithmic asymptotics, obtained from classical results on
additive and multiplicative functions. Coefficient formulas also give
irrationality and nonvanishing, determine the poles and residues of
generating functions for normalized Mellin moments, and imply that
these generating functions are not D-finite. We conclude with
applications to integrals involving zeta, Dirichlet $L$, and cusp-form
$L$-functions.
\end{abstract}
\noindent\textbf{Keywords.} Rogers--Ramanujan continued fraction;
coefficient arithmetic; sign density; arithmetic progressions; Mellin transform.

\smallskip
\noindent\textbf{2020 Mathematics Subject Classification.}
Primary 11A55, 11N37; Secondary 11N60, 11M06, 44A20, 11B37.

\section{Introduction}\label{sec:introduction}
For $0<q<1$, the Rogers--Ramanujan continued fraction is
\begin{equation}\label{intro:products}
 R(q)=\cfrac{q^{1/5}}{1+\cfrac{q}{1+\cfrac{q^2}{1+\cdots}}}
 =q^{1/5}F(q),\qquad
 F(q)=\prod_{j\ge1}(1-q^j)^{\chi(j)}=\sum_{m\ge0}P_mq^m,
\end{equation}
where $\chi(n)=(n/5)$ is the real primitive character modulo $5$ and
$\phi=(1+\sqrt5)/2$; the product is classical
\cite{Rogers,AndrewsBerndt}. Put $H(z)=zF(z^5)$ and define
\begin{equation}\label{intro:kernels}
 D(z)=\frac{\sqrt5H(z)}{\phi+H(z)}=\sum_{k\ge1}d_kz^k,
 \qquad V(z)=\log(1+H(z)/\phi)=\sum_{k\ge1}c_kz^k.
\end{equation}
The logarithm is fixed by $V(0)=0$; Lemma~\ref{coeff:analytic} proves
holomorphy throughout the unit disk. We study the arithmetic of $d_k,c_k$
and its consequences for Mellin moments. Our main results concern the
signs of the logarithmic coefficients $c_k$.

The ordinary coefficients $P_m$ have a classical periodic sign pattern,
apart from four zeros \cite{Andrews,Hirschhorn,RichmondSzekeres}.
Sign patterns for related infinite products and powers are studied in
\cite{ChanYesilyurt,ChernTang,BaruahSarma}. The logarithmic coefficients
in \eqref{intro:kernels} behave differently: both signs occur in every
reduced residue class, while negative signs become rare on the levels
$v_5(n)=a$ as $a$ increases. We first classify the signs in every arithmetic progression and then
quantify the density and first occurrence of negative coefficients.
All densities below are natural densities of sets of positive integers.

The first theorem includes progressions whose residue class is not
coprime to the modulus. For $5\nmid m$ put
\begin{equation}\label{intro:ratio}
 r_m=\frac{\sum_{d\mid m}d^{-1}}{\sum_{d\mid m}\chi(d)d^{-1}}.
\end{equation}
For primes $p$ with $\chi(p)=-1$, and for $j\ge0$, define
\begin{equation}\label{intro:local-factor}
 \rho_p(j)=\frac{p+1}{p-1}
     \frac{1-p^{-j-1}}{1-(-1/p)^{j+1}}.
\end{equation}
The sign thresholds are
\begin{equation}\label{intro:thresholds}
 T_{0,\epsilon}=\sqrt5,\qquad
 T_{a,\epsilon}=\frac{5^a\sqrt5-5\epsilon}{6}\quad(a\ge1),
 \qquad \epsilon\in\{-1,1\}.
\end{equation}
Given $M\ge1$ and $b\in\ZZ$, take $v_p(0)=+\infty$, write
$f_p=v_p(M)$, $e_p=v_p(b)$, and set
\begin{equation}\label{intro:obstruction}
 K_{M,b}=\prod_{\substack{p\mid M\\\chi(p)=-1}} k_p,
 \qquad
 k_p=\begin{cases}
 \rho_p(e_p),&e_p<f_p,\\
 \rho_p(2\lceil f_p/2\rceil),&e_p\ge f_p.
 \end{cases}
\end{equation}
An empty product is $1$. In particular, $K_{M,b}$ is a positive rational number.

\begin{theorem}[Signs in arithmetic progressions]\label{thm:congruence}
In every progression $n\equiv b\pmod M$, the positive and negative
coefficient sets have natural densities. The set of indices with negative coefficients always has positive
density. Write $v=v_5(M)$.
\begin{enumerate}
\item If $5^v\mid b$, including $v=0$, both signs have positive density.
\item Otherwise, put $a=v_5(b)<v$ and $\epsilon=\chi(b/5^a)$.
If $K_{M,b}<T_{a,\epsilon}$, both signs have positive density.
If $K_{M,b}>T_{a,\epsilon}$, every coefficient in the progression is negative.
\end{enumerate}
Equality in this comparison is impossible. Thus a progression contains
a positive coefficient if and only if its positive coefficients have
positive natural density.
\end{theorem}

The second theorem concerns the density of negative coefficients among
indices with a prescribed exponent of $5$, and the least index at which
a negative coefficient occurs.

\begin{theorem}[Density and first occurrence of negative coefficients]\label{thm:rarity}
For every integer $a\ge0$, the relative density
\[
 \delta_a^-=
 \lim_{X\to\infty}
 \frac{\#\{m\le X:5\nmid m,\ c_{5^am}<0\}}
      {\#\{m\le X:5\nmid m\}}
\]
exists and lies strictly between $0$ and $1$. Define
\[
 m_a=\min\{m\ge1:5\nmid m,\ c_{5^am}<0\},\qquad
 \kappa=\frac{5\log\phi}{12e^\gamma},
\]
where $\gamma$ is Euler's constant. Then
\begin{equation}\label{intro:rarity-limits}
 \lim_{a\to\infty}\frac{\log\log(1/\delta_a^-)}{5^a}
 =\lim_{a\to\infty}\frac{\log\log m_a}{5^a}=\kappa.
\end{equation}
Equivalently,
\[
 \delta_a^-=\exp\{-\exp((\kappa+o(1))5^a)\},\qquad
 m_a=\exp\{\exp((\kappa+o(1))5^a)\}.
\]
The least index of a negative coefficient on the level $v_5(n)=a$,
$5^am_a$, satisfies the same double-logarithmic limit. The absolute density of negative coefficients
on the level $v_5(n)=a$ is $4\delta_a^-/5^{a+1}$.
\end{theorem}

The proofs use the divisor formula for $c_n$ to reduce the sign problem
to the positive multiplicative ratio \eqref{intro:ratio}. The existence
and continuity of its limiting distribution are applications of classical
Erd\H{o}s--Wintner theory \cite{TenenbaumVerwee}; the passage to
progressions is made explicit by approximation using finitely many primes. The local
minima \eqref{intro:obstruction} give the necessary and sufficient
condition in Theorem~\ref{thm:congruence}.
For Theorem~\ref{thm:rarity}, the relevant prime product is
\[
 \prod_{\substack{p\le y\\\chi(p)=-1}}\frac{p+1}{p-1}
 \sim\frac{2e^\gamma}{\sqrt5\log\phi}\log y.
\]
Moment estimates and prescribed divisibility by small primes give
the tail bounds; compare Weingartner's sharper results for $\sigma(n)/n$
and $n/\varphi(n)$ \cite{Weingartner}. The maximal-order estimate is a
specialization of T\'oth--Wirsing \cite{TothWirsing}. The contribution
here is the complete sign classification for these logarithmic
coefficients and the leading double-logarithmic asymptotics in
Theorem~\ref{thm:rarity}. The latter theorem gives no asymptotic
equivalent for $\delta_a^-$ or $m_a$, and does not assert
$m_a\sim1/\delta_a^-$.

The coefficient arithmetic also controls the analytic transforms.
Every $d_k$ is a quadratic irrational algebraic integer with an explicit
minimal polynomial, and every $c_k$ is quadratic irrational.
Their nonvanishing determines which residues occur in the generating
functions of normalized Mellin moments; the exact pole classification
is Theorem~\ref{thm:main}. The remainder estimate in
Lemma~\ref{mellin:transfer} is necessary for the ordinary moments,
whose formal Dirichlet series fails the term test at every fixed order.
It also yields non-D-finiteness and non-P-recursiveness, using the
classical equivalence in \cite{Stanley} and comparison of successive
exponential terms, as in \cite[Theorem~6.1]{BellEtAl}.

The integral formulas in Section~\ref{sec:logarithmic} are applications
of Ramanujan's factorizations, the character-weighted divisor identity
recorded by Murty \cite[Theorem~5.2 and p.~83]{Murty}, and the quintic
Eisenstein identities of Cooper--Toh \cite[Theorem~2.6]{CooperToh}.
They are included to connect the coefficient results with explicit
Mellin transforms; no irrationality or linear independence of their
odd special values is asserted.

Section~\ref{sec:coefficients} proves the coefficient formulas.
Sections~\ref{sec:signs} and \ref{sec:rarity} prove
Theorems~\ref{thm:congruence} and \ref{thm:rarity}, respectively.
Sections~\ref{sec:mellin}--\ref{sec:moments} develop the Mellin
applications. The appendices contain the reflection bounds, vertical
decay, and general special-value formulas.

We use $(a;q)_\infty=\prod_{j\ge0}(1-aq^j)$ and the standard
meromorphic continuations of $\zeta$ and $L(s,\chi)$.
All powers of positive real numbers use real logarithms. For complex
arguments we define $R(e^{-t})=e^{-t/5}F(e^{-t})$, which is holomorphic on
$\Re t>0$. The classical reciprocity identity
\begin{equation}\label{intro:reciprocity}
 (\phi+R(e^{-t}))(\phi+R(e^{-4\pi^2/t}))=\phi\sqrt5
\end{equation}
holds for $t>0$ \cite{WatsonIX,AndrewsBerndt}, and extends to that
half-plane by the identity theorem.

\section{Coefficient arithmetic}\label{sec:coefficients}

Write
\[
 F(q)=\sum_{m\ge0}P_mq^m,\qquad
 D(z)=\sum_{k\ge1}d_kz^k,\qquad
 V(z)=\sum_{N\ge1}c_Nz^N.
\]
We first justify these expansions throughout the unit disk. The
nonvanishing results below will determine which candidate poles of the
generating functions for Mellin moments actually occur.

\begin{lemma}\label{coeff:analytic}
The functions $H,D,V$ are holomorphic on $|z|<1$, where $V(0)=0$ fixes
the logarithm. In this disk,
\begin{equation}\label{coeff:products}
 1-\phi D(z)
 =\prod_{n\ge1}\frac{1-z^n/\phi+z^{2n}}{1+\phi z^n+z^{2n}}
 =\exp\!\left(-\sqrt5\sum_{n\ge1}
       \left(\sum_{e\mid n}\frac{\chi(e)}e\right)z^n\right).
\end{equation}
\end{lemma}

\begin{proof}
The product defining $F$ makes $H$ holomorphic in the disk. Put
\[
 Q_1(z)=\prod_{n\ge1}(1-z^n/\phi+z^{2n}),\qquad
 Q_2(z)=\prod_{n\ge1}(1+\phi z^n+z^{2n}).
\]
If $\omega=e^{2\pi i/5}$, their respective factors are
$(1-\omega z^n)(1-\omega^{-1}z^n)$ and
$(1-\omega^2z^n)(1-\omega^{-2}z^n)$. Both products converge locally
uniformly and are nonzero in the disk.
Ramanujan's factorizations
\cite[Corollary~2, equations~(1.8)--(1.9)]{ChanChanLiu}, also recorded
in \cite[pp.~21--24]{AndrewsBerndt}, give, for $q=z^5\in(0,1)$,
\begin{equation}\label{coeff:factorizations}
 \begin{aligned}
 \frac{1+R(q)/\phi}{\sqrt{R(q)}}
 &=q^{-1/10}\sqrt{\frac{(q;q)_\infty}{(q^5;q^5)_\infty}}\,Q_1(z)^{-1},\\
 \frac{1-\phi R(q)}{\sqrt{R(q)}}
 &=q^{-1/10}\sqrt{\frac{(q;q)_\infty}{(q^5;q^5)_\infty}}\,Q_2(z)^{-1}.
 \end{aligned}
\end{equation}
Their ratio, followed by the identity theorem, yields
$(1+H/\phi)Q_1=(1-\phi H)Q_2$ throughout the disk.
If $H=-\phi$, this identity would imply
$0=(1+\phi^2)Q_2$. Thus $1+H/\phi$ is nonzero, proving the assertions
about $D$ and $V$ and the first equality in \eqref{coeff:products}.
Expand the product logarithms at zero and use
$2\cos(2\pi e/5)-2\cos(4\pi e/5)=\sqrt5\chi(e)$ to obtain the second.
\end{proof}

Define
\begin{equation}\label{coeff:partition-product}
 E(z)=\frac{(z^{25};z^{25})_\infty}{(z;z)_\infty}
     =\prod_{\substack{j\ge1\\25\nmid j}}(1-z^j)^{-1}
     =\sum_{n\ge0}p_{25}(n)z^n.
\end{equation}
Here $p_{25}(n)$ counts partitions with no part divisible by $25$;
we set $p_{25}(n)=0$ for $n<0$. In particular $p_{25}(n)>0$ for $n\ge0$.

\begin{theorem}\label{coeff:arithmetic}
For $k\ge1$, put
$p_k=p_{25}(k-1)$ and $b_k=[z^{k-2}]F(z^5)E(z)$, with coefficients
of negative degree interpreted as zero. Then $b_1=0$ and $b_k\ge1$
for $k\ge2$. Every $d_k$ is a quadratic irrational algebraic integer,
given by
\begin{equation}\label{coeff:dformula}
 d_k=\frac{5p_k-\sqrt5(p_k+2b_k)}2.
\end{equation}
Its minimal polynomial over $\QQ$ is
\begin{equation}\label{coeff:minpoly}
 X^2-5p_kX+5(p_k^2-p_kb_k-b_k^2).
\end{equation}
\end{theorem}

\begin{proof}
Let a bar denote conjugation in $\QQ(\sqrt5)$, fixing $z$.
Since $H$ has integer coefficients and $\bar\phi=-\phi^{-1}$,
\[
 D+\bar D=\frac{5H}{1-H-H^2}.
\]
Ramanujan's identity, first proved by Watson \cite{WatsonVII} and
recorded in \cite[equation~(6.1)]{BerndtRebak}, is
\begin{equation}\label{coeff:ramanujan}
 H(z)^{-1}-1-H(z)
   =\frac{(z;z)_\infty}{z(z^{25};z^{25})_\infty}.
\end{equation}
Consequently
\begin{equation}\label{coeff:trace}
 D+\bar D=5zE(z),\qquad
 D=\sqrt5\bigl(\phi^{-1}-H(z)\bigr)zE(z).
\end{equation}
The second identity gives $d_k=\sqrt5(\phi^{-1}p_k-b_k)$, which is
\eqref{coeff:dformula}. Cancellation of the defining products gives
\begin{equation}\label{coeff:positive-product}
 F(z^5)E(z)
 =\prod_{\substack{n\ge1\\5\nmid n}}(1-z^n)^{-1}
  \prod_{j\ge0}(1-z^{25j+10})^{-2}(1-z^{25j+15})^{-2}.
\end{equation}
All coefficients are nonnegative integers, and the factor
$(1-z)^{-1}$ makes each coefficient of nonnegative degree positive.
This proves the assertions about $b_k$.
The conjugate of \eqref{coeff:dformula} has the opposite sign before
$\sqrt5$. Its trace and norm give \eqref{coeff:minpoly}, whose
discriminant $5(p_k+2b_k)^2$ is a nonzero nonsquare in $\QQ$.
The polynomial is therefore irreducible and monic with integer coefficients.
\end{proof}

The trace identity alone proves $d_k\ne0$: it gives
$d_k+\bar d_k=5p_{25}(k-1)>0$.

\begin{proposition}\label{coeff:logarithmic}
For $N\ge1$,
\begin{align}\label{coeff:divisors}
 c_N={}&\frac{\sqrt5}{2}\sum_{e\mid N}\frac{\chi(e)}e
       -\frac12\sum_{e\mid N}\frac1e \notag\\
      &-\frac{\mathbf1_{5\mid N}}2
           \sum_{e\mid N/5}\frac{\chi(N/(5e))}{e}
       +\frac{\mathbf1_{25\mid N}}2\sum_{e\mid N/25}\frac1e.
\end{align}
The indicated sums are omitted when their divisor arguments are not
integers. Every $c_N$ belongs to $\QQ(\sqrt5)\setminus\QQ$, and
\begin{equation}\label{coeff:log-bound}
 |c_N|\le\frac{\sqrt5+3}{2}(1+\log N).
\end{equation}
\end{proposition}

\begin{proof}
Taking the logarithm of the first identity in
\eqref{coeff:factorizations} gives
\[
 V(z)=\frac12\log F(z^5)
      +\frac12\log\frac{(z^5;z^5)_\infty}{(z^{25};z^{25})_\infty}
      -\log Q_1(z).
\]
Expanding these logarithms and using
$2\cos(2\pi k/5)=(\sqrt5\chi(k)-1+5\mathbf1_{5\mid k})/2$
yields
\begin{align}\label{coeff:lambert}
 2V(z)={}&\sqrt5\sum_{k\ge1}\frac{\chi(k)}k\frac{z^k}{1-z^k}
          -\sum_{k\ge1}\frac1k\frac{z^k}{1-z^k} \notag\\
         &-\sum_{j,k\ge1}\frac{\chi(j)}kz^{5jk}
          +\sum_{k\ge1}\frac1k\frac{z^{25k}}{1-z^{25k}}.
\end{align}
All sums converge absolutely on compact subdisks, so coefficient
extraction proves \eqref{coeff:divisors}.
The coefficient of $\sqrt5$ there is half the positive rational number
\[
 \sum_{e\mid N}\frac{\chi(e)}e
 =\prod_{p^a\parallel N}\left(\sum_{j=0}^a\frac{\chi(p)^j}{p^j}\right)>0.
\]
Indeed, the local factors for $\chi(p)=0,1$ are positive, and for
$\chi(p)=-1$ the factor is
$(1-(-1/p)^{a+1})/(1+1/p)>0$.
All other terms in \eqref{coeff:divisors} are rational. This proves
irrationality, while bounding each divisor sum by
$\sum_{j=1}^N j^{-1}\le1+\log N$ proves \eqref{coeff:log-bound}.
\end{proof}

\begin{proposition}\label{coeff:support}
The nonzero coefficients of $H$ occur precisely at
\[
 5m+1\quad(m\ge0,\quad m\notin\{3,8,13,23\});
\]
every positive-degree coefficient of $D$ and $V$ is nonzero.
For each $0<r<1$ one has $P_m=O_r(r^{-m})$ and
$d_k=O_r(r^{-k})$, whereas $c_N=O(1+\log N)$.
The formal series $\sum_{m\ge0}P_m(5m+1)^{-2s}$ fails the term test
for every fixed $s\in\CC$.
\end{proposition}

\begin{proof}
The assertions about $D,V$ follow from the preceding results, and the
subexponential bounds follow from Cauchy's estimates.
For $F$, Hirschhorn's five-dissection \cite{Hirschhorn} proves
$P_m=0$ exactly for $m\in\{3,8,13,23\}$; the remaining signs are
$+,-,+,-,-$ according to $m\bmod5$.
The normalization matters: Hirschhorn's expansion of the reciprocal
normalized continued fraction is precisely our $F(q)$.
The asymptotic formula of Richmond and Szekeres
\cite{RichmondSzekeres}, reproduced in \cite{Hirschhorn}, is
\begin{equation}\label{coeff:ordinary-asymptotic}
 P_m=\frac{\sqrt2}{5^{3/4}}m^{-3/4}e^{4\pi\sqrt{m/125}}
 \left\{\cos\!\left(\frac{4\pi m}{5}+\frac{3\pi}{25}\right)
                  +O(m^{-1/2})\right\}.
\end{equation}
The cosine takes five nonzero values, each of absolute value at least
$\sin(\pi/50)$. Thus the magnitude of $P_m$ grows faster than every
fixed power of $m$, proving the last assertion.
\end{proof}

We therefore use finite expansions with controlled remainders
for the ordinary moments.

\section{Signs in arithmetic progressions}\label{sec:signs}

The divisor formula \eqref{coeff:divisors} reduces the signs to a
positive multiplicative ratio. For $5\nmid m$, write
\begin{equation}\label{sign:ratios}
 A_m=\sum_{d\mid m}\frac{\chi(d)}d>0,\qquad
 S_m=\sum_{d\mid m}\frac1d,\qquad r_m=S_m/A_m.
\end{equation}
We first establish the local formulas used in the main theorems.

\begin{proposition}[Coefficient signs]\label{sign:local}
For $5\nmid m$,
\begin{align}
 c_m&=\frac{A_m}{2}(\sqrt5-r_m),\label{sign:at-m}\\
 c_{5^a m}&=\frac{A_m}{2}
  \left\{\sqrt5-5^{-a}\bigl(5\chi(m)+6r_m\bigr)\right\}
  \qquad(a\ge1).\label{sign:five-adic}
\end{align}
The local factor of $r_m$ is $1$ at primes with $\chi(p)=1$ and $\rho_p(v_p(m))$
at primes with $\chi(p)=-1$. Consequently
\begin{equation}\label{sign:threshold}
 \operatorname{sgn}c_{5^a m}
   =\operatorname{sgn}(T_{a,\chi(m)}-r_m)\qquad(a\ge0).
\end{equation}
For fixed $m$, the sequence $(c_{5^a m})_{a\ge1}$ is strictly increasing
to $\sqrt5 A_m/2$ and changes sign at most once. Moreover,
$c_m<0$ whenever $m$ is even and $5\nmid m$.
\end{proposition}

\begin{proof}
For $p\ne5$, summing finite geometric series gives
\begin{equation}\label{sign:local-factor}
 r_{p^j}=
 \begin{cases}
 1,&\chi(p)=1,\\
 \displaystyle\rho_p(j)=\frac{p+1}{p-1}
       \frac{1-p^{-j-1}}{1-(-1/p)^{j+1}},&\chi(p)=-1.
 \end{cases}
\end{equation}
Put $W_p=(p+1)/(p-1)$ for $\chi(p)=-1$. Then $\rho_p(0)=1$,
all odd-exponent values equal $W_p$, and the even-exponent values
increase strictly to $W_p$. Hence $r_m\ge1$.
Equation~\eqref{coeff:divisors} immediately gives \eqref{sign:at-m}.
For $N=5^a m$, $a\ge1$, the mixed divisor sum in that equation is
$5^{1-a}\chi(m)A_m$, and the difference of its two unweighted sums
is $6\cdot5^{-a}S_m$. This proves \eqref{sign:five-adic}.
The number $5\chi(m)+6r_m$ is positive, so the monotonicity and limit
follow. Rationality of $r_m$ excludes equality with the irrational
threshold. Finally, if $2^e\parallel m$, then
$r_m\ge\rho_2(e)\ge\rho_2(2)=7/3>\sqrt5$.
\end{proof}

The monotonicity statement starts at $a=1$; the transition from $c_m$
to $c_{5m}$ need not be increasing. We now prove the distribution
statement for arbitrary congruence classes.
Write $\mathcal P_-=\{p:\chi(p)=-1\}$. For a fixed progression
$b\pmod M$, retain $f_p=v_p(M)$ and $e_p=v_p(b)$ at primes dividing
$M$. Independently for $p\in\mathcal P_-$, let
\begin{equation}\label{sign:geometric}
 \mathbb P(G_p=j)=(1-p^{-1})p^{-j}\qquad(j\ge0),\qquad
 V_p=\begin{cases}
 e_p,&p\mid M,\ e_p<f_p,\\
 f_p+G_p,&p\mid M,\ e_p\ge f_p,\\
 G_p,&p\nmid M.
 \end{cases}
\end{equation}
Define
\begin{equation}\label{sign:random-product}
 \mathcal R_{M,b}=\prod_{p\in\mathcal P_-}\rho_p(V_p),\qquad
 \mathcal F_{M,b}(t)=\mathbb P(\mathcal R_{M,b}<t).
\end{equation}

\begin{lemma}[Limiting distribution and its support]\label{sign:law}
The product \eqref{sign:random-product} converges almost surely. Its
distribution is continuous with support $[K_{M,b},\infty)$, and its
distribution function is strictly increasing above $K_{M,b}$.
Normalized by the number of integers in the progression, the values
$r_{n/5^{v_5(n)}}$, $n\equiv b\pmod M$, have this limiting distribution.
On each compatible set with fixed $v_5(n)$, prescribing the unit part
modulo $5$ leaves the limiting distribution unchanged. If all four nonzero unit
residues are allowed, each has relative frequency $1/4$, and the
conditional limiting distribution is the same in every residue class.
\end{lemma}

\begin{proof}
Write $w_p=\log W_p\le3/p$. Outside the finite set of prime divisors
of $M$,
\[
 0\le\log\rho_p(G_p)\le w_p\mathbf1_{G_p\ge1},\qquad
 \mathbb E\log\rho_p(G_p)\le3/p^2.
\]
Thus the nonnegative logarithmic series has finite expectation and
converges almost surely. The Chinese remainder theorem gives the
independent distributions \eqref{sign:geometric} for every finite set of
valuations; the probabilities of the omitted valuation ranges tend to zero geometrically.
For $y$ beyond all prime divisors of $M$,
\begin{equation}\label{sign:mean-tail}
 \frac M X\sum_{\substack{n\le X\\n\equiv b\ (M)}}
 \sum_{\substack{p>y\\p\in\mathcal P_-,\ p\mid n}}w_p
 \ll\sum_{p>y}p^{-2}+\frac{M\log X}{X}.
\end{equation}
Indeed, the count of multiples of $p\nmid M$ in the progression is
$X/(Mp)+O(1)$, and $\sum_{p\le X}w_p\ll\log X$.
Approximation using finitely many primes and Markov's inequality identify the
limiting distribution at every continuity point.

To prove continuity, use the classical Erd\H{o}s--Wintner theorem and
L\'evy's criterion, as recorded in \cite[Section~1]{TenenbaumVerwee}, for
the additive function with prime-power values $\log\rho_p(j)$ at
$p\in\mathcal P_-$, $p\nmid M$, and zero elsewhere. Its first and
second prime moment series converge. Also
\begin{equation}\label{sign:inert-divergence}
 \sum_{p\in\mathcal P_-}\frac1p=\infty.
\end{equation}
For example, the Euler products give
$\log\zeta(\sigma)-\log L(\sigma,\chi)
=2\sum_{p\in\mathcal P_-}p^{-\sigma}+O(1)$ as $\sigma\downarrow1$,
while $L(1,\chi)=2\log\phi/\sqrt5>0$
\cite[Chapter~4]{Washington}. The continuity criterion applies;
convolution with the finitely many constrained local distributions preserves
continuity.

At $p\mid M$, the smallest allowed factor is exactly the factor
appearing in $K_{M,b}$, and it has positive probability. This proves
$\mathcal R_{M,b}\ge K_{M,b}$ and the corresponding pointwise lower
bound on integers. Assign these finitely many minimizing valuations.
At the remaining primes in $\mathcal P_-$, the available logarithmic increments
include $0$ and $w_p$. Since $w_p\to0$ and $\sum w_p=\infty$, finitely
many increments from arbitrarily large primes can reach any prescribed
positive interval. Prescribe valuation $1$ at these primes and $0$ at
the other primes up to the last selected one. The prescription has
positive probability, and the remaining tail has arbitrarily small
expectation. Independence and Markov's inequality therefore give
positive mass to every open interval above $\log K_{M,b}$, and also
to every right neighborhood of that endpoint. This proves the support
and strict increase assertions.

For the assertion about the unit part, fix $a=v_5(n)$ and prescribe
$n\equiv5^a u\pmod{5^{a+1}}$ with a compatible
$u\in\{1,2,3,4\}$. Adding this condition to each finite CRT
specification leaves the valuation distributions at primes in $\mathcal P_-$ unchanged. When all four
units are permitted, their frequencies are equal.
The tail estimate \eqref{sign:mean-tail} transfers this independence
to the limiting product.
\end{proof}

\begin{proof}[Proof of Theorem~\ref{thm:congruence}]
Let $d_+(M,b)$ denote the absolute natural density of positive
coefficients in the progression, and put $v=v_5(M)$.
If $a=v_5(b)<v$, both $a$ and $\epsilon=\chi(b/5^a)$ are fixed, so
\begin{equation}\label{sign:fixed-density}
 d_+(M,b)=\frac1M\mathcal F_{M,b}(T_{a,\epsilon}).
\end{equation}
If $5^v\mid b$, the relative frequency of $v_5(n)=a\ge v$ is
$4/5^{a-v+1}$, and $\chi(n/5^a)$ takes each of the values $1$ and $-1$
with relative frequency $1/2$. Hence
\begin{equation}\label{sign:free-density}
 d_+(M,b)=\frac1M\sum_{a=v}^\infty\frac4{5^{a-v+1}}
 \frac{\mathcal F_{M,b}(T_{a,1})+
       \mathcal F_{M,b}(T_{a,-1})}{2}.
\end{equation}
For this infinite mixture, first retain $a\le A$; the omitted absolute
density is $1/(M5^{A-v+1})$, tending to zero. Since coefficients
never vanish, $d_-(M,b)=1/M-d_+(M,b)$.

The unbounded support gives a positive probability of a negative
coefficient at every finite threshold. When $v_5(b)<v$, comparison of the lower endpoint of the support
with $T_{a,\epsilon}$ gives the conclusion; when
$K_{M,b}>T_{a,\epsilon}$, the pointwise bound makes every coefficient
negative. Equality is excluded by rationality of $K_{M,b}$ and
irrationality of the threshold. When $5^v\mid b$, the thresholds tend to infinity, so some terms
in \eqref{sign:free-density} give a positive probability of a positive
coefficient. This proves the theorem.
\end{proof}

\begin{corollary}[Reduced residue classes]\label{sign:distribution}
In every reduced residue class modulo every positive integer, both
coefficient signs have positive natural density. The sequence
$(\operatorname{sgn}c_n)_{n\ge1}$ is not eventually periodic.
\end{corollary}
\begin{proof}
For a reduced class $K_{M,b}=1$, and every threshold exceeds $1$.
Apply Theorem~\ref{thm:congruence}. Eventual periodicity of period
$T$ would make the signs eventually constant on $1\pmod T$,
contradicting the two density assertions.
\end{proof}

For example, take $M=50$. For $b=20$ and $b=10$ one has
$K_{50,b}=\rho_2(2)=7/3$ and $v_5(b)=1$, whereas
$\chi(20/5)=1$ and $\chi(10/5)=-1$. Since
\[
 \frac{5\sqrt5-5}{6}<\frac73<\frac{5\sqrt5+5}{6},
\]
every coefficient $c_{50j+20}$ is negative for $j\ge0$, while the
progression $50j+10$ contains both signs with positive natural density.

\section{Density and first occurrence of negative coefficients}\label{sec:rarity}

For the unrestricted valuations at primes in $\mathcal P_-$ write
\[
 \mathcal R=\prod_{p\in\mathcal P_-}\rho_p(G_p),\qquad
 \mathcal D(t)=\mathbb P(\mathcal R>t).
\]
Lemma~\ref{sign:law} gives a continuous distribution with support $[1,\infty)$,
governing $r_m$ among $5\nmid m$ and in each nonzero residue class
modulo $5$. Define
\begin{equation}\label{rare:constant}
 \mathcal E(y)=\prod_{\substack{p\le y\\p\in\mathcal P_-}}
                    \frac{p+1}{p-1},\qquad
 \mathcal C=\frac{4e^\gamma}{5L(1,\chi)}
           =\frac{2e^\gamma}{\sqrt5\log\phi}.
\end{equation}
For $y\ge5$, the finite-product identity
\[
 \mathcal E(y)=\frac45\prod_{p\le y}(1-p^{-1})^{-1}
                         \prod_{p\le y}(1-\chi(p)p^{-1})
\]
and the classical prime number theorem in progressions and Mertens
product theorem imply
\begin{equation}\label{rare:mertens}
 \mathcal E(y)\sim\mathcal C\log y.
\end{equation}
The Euler product for the character here converges in increasing prime order
to $L(1,\chi)^{-1}$; absolute convergence is not needed.

\begin{lemma}[Tail of the limiting distribution]\label{rare:tail}
As $t\to\infty$,
\begin{equation}\label{rare:tail-limit}
 \frac{\log\log(1/\mathcal D(t))}{t}\longrightarrow
 \frac1{\mathcal C}.
\end{equation}
\end{lemma}
\begin{proof}
We use moment estimates and products over small primes, as in
the study of divisor-ratio distributions; compare \cite{Weingartner}.
For real $s\to\infty$, primes $p\le s$ contribute at most
$\mathcal E(s)^s$ to $\mathbb E\mathcal R^s$. For $p>s$,
\[
 \log\mathbb E\rho_p(G_p)^s
 \le\frac{e^{s w_p}-1}{p}\ll\frac{s}{p^2},\qquad
 w_p=\log\frac{p+1}{p-1}\le\frac3p.
\]
Since $s\sum_{p>s}p^{-2}\ll1$, monotone convergence of the finite
products gives
\begin{equation}\label{rare:moment-bound}
 \mathbb E\mathcal R^s\le\mathcal E(s)^s\exp(O(1)).
\end{equation}
Fix $0<\varepsilon<1$ and set
$s=\exp((1-\varepsilon)t/\mathcal C)$. Equation~\eqref{rare:mertens}
gives $\mathcal E(s)/t\le1-\varepsilon/2$ for large $t$.
Markov's inequality therefore gives
\[
 \mathcal D(t)\le\exp(O(1))\{\mathcal E(s)/t\}^s
 \le\exp(-c_\varepsilon s+O(1)),
\]
proving that the liminf in \eqref{rare:tail-limit} is at least
$(1-\varepsilon)/\mathcal C$.

For the reverse inequality put
$y=\exp((1+\varepsilon)t/\mathcal C)$ and prescribe $G_p=1$ for
all primes $p\le y$ in $\mathcal P_-$. This forces $\mathcal R\ge\mathcal E(y)>t$.
The event has probability
\[
 \prod_{\substack{p\le y\\p\in\mathcal P_-}}\frac{1-p^{-1}}p,
\]
whose negative logarithm is
$\sum_{p\le y,p\in\mathcal P_-}\log p
-\sum_{p\le y,p\in\mathcal P_-}\log(1-p^{-1})\ll y$.
Chebyshev's bound suffices for this estimate. Hence
$\mathcal D(t)\ge\exp(-O(y))$, giving a limsup at most
$(1+\varepsilon)/\mathcal C$. Let $\varepsilon\downarrow0$.
\end{proof}

\begin{lemma}[Maximal order]\label{rare:maximal-order}
As $m\to\infty$ with $5\nmid m$,
\begin{equation}\label{rare:maximal-limit}
 r_m\le(\mathcal C+o(1))\log\log m,\qquad
 \limsup_{\substack{m\to\infty\\5\nmid m}}
   \frac{r_m}{\log\log m}=\mathcal C.
\end{equation}
\end{lemma}
\begin{proof}
This is a specialization of \cite[Theorems~2--3]{TothWirsing}; we
include the direct argument. Split the prime factors of $m$ in $\mathcal P_-$
at $x=\log m$. The small primes contribute at most $\mathcal E(x)$.
There are at most $\log m/\log\log m$ distinct prime factors above
$x$, so their total logarithmic contribution is at most
$3/\log\log m$. Equation~\eqref{rare:mertens} proves the upper bound.
For the lower bound, the product
\[
 P(y)=\prod_{\substack{p\le y\\p\in\mathcal P_-}}p
\]
satisfies $r_{P(y)}=\mathcal E(y)$ and
$\log P(y)\sim y/2$ by the prime number theorem in progressions.
Thus $r_{P(y)}/\log\log P(y)\to\mathcal C$.
\end{proof}

\begin{proof}[Proof of Theorem~\ref{thm:rarity}]
The limiting distribution of $r_m$ is the same for $\chi(m)=1$
and $\chi(m)=-1$, so
\begin{equation}\label{rare:exact-density}
 \delta_0^-=\mathcal D(\sqrt5),\qquad
 \delta_a^-=
 \frac{\mathcal D((5^a\sqrt5-5)/6)+
       \mathcal D((5^a\sqrt5+5)/6)}2\quad(a\ge1).
\end{equation}
Continuity and full support give $0<\delta_a^-<1$ at every finite
level. Both arguments are $(\sqrt5/6)5^a+O(1)$, so monotonicity of $\mathcal D$ and Lemma~\ref{rare:tail} yield
\[
 \lim_{a\to\infty}5^{-a}\log\log(1/\delta_a^-)
 =\frac{\sqrt5}{6\mathcal C}=\frac{5\log\phi}{12e^\gamma}.
\]
The absolute density of negative indices on $v_5(n)=a$ is
$4\delta_a^-/5^{a+1}$, by $n=5^a m$ with $5\nmid m$.

For the first occurrence, let $t_a=(\sqrt5/6)5^a$ for $a\ge1$.
Equation~\eqref{sign:threshold} gives
\[
 c_{5^a m}<0\ \Longrightarrow\ r_m>t_a-5/6,
 \qquad
 r_m>t_a+5/6\ \Longrightarrow\ c_{5^a m}<0.
\]
Each $m_a$ exists because the ratio is unbounded. Also $m_a\to\infty$:
every fixed finite set of $m$ has bounded ratios and eventually lies
below both thresholds. Lemma~\ref{rare:maximal-order} therefore gives
\[
 \liminf_{a\to\infty}5^{-a}\log\log m_a
 \ge\frac{\sqrt5}{6\mathcal C}.
\]
For the reverse bound fix $\varepsilon>0$ and put
$y_a=\exp((1+\varepsilon)t_a/\mathcal C)$.
For large $a$, $\mathcal E(y_a)>t_a+5/6$, so $m_a\le P(y_a)$,
regardless of the value of $\chi(P(y_a))$. Since
\[
 \log\log P(y_a)=\log y_a+O(1)
                  =(1+\varepsilon)t_a/\mathcal C+O(1),
\]
the limsup is at most $(1+\varepsilon)\sqrt5/(6\mathcal C)$.
Let $\varepsilon\downarrow0$. Multiplication of $m_a$ by $5^a$
does not change the double-logarithmic limit for the actual index.
\end{proof}

\section{Mellin transforms and coefficient expansions}
\label{sec:mellin}
A function is D-finite if it satisfies a nonzero homogeneous linear differential equation with rational-function coefficients. A sequence is P-recursive if it satisfies a nonzero finite linear recurrence with polynomial coefficients for all sufficiently large indices.

\begin{lemma}\label{mellin:transfer}
Let $W(z)=\sum_{k\ge1}a_kz^k$ be holomorphic in $|z|<1$ and admit a continuous extension to $[0,1]$. Set
\[
 N(s)=\frac1{\Gamma(2s)}\int_0^\infty t^{2s-1}W(e^{-t})\,dt
 \qquad(\Re s>0).
\]
For every fixed $K\ge1$ and $\sigma=\Re s>0$,
\begin{equation}\label{mellin:remainder}
 \left|N(s)-\sum_{k=1}^K a_kk^{-2s}\right|
 \le C_K(K+1)^{-2\sigma}\frac{\Gamma(2\sigma)}{|\Gamma(2s)|}.
\end{equation}
Each fixed derivative of the remainder is $O_{K,j}((K+1)^{-2s})$ as real $s\to+\infty$.
The generating function $C(z)=\sum_{n\ge1}N(n)z^n$ continues meromorphically to $\CC$, with exactly the simple poles $k^2$ for which $a_k\ne0$, of residues $-k^2a_k$.
If infinitely many $a_k$ are nonzero, then $C$ and $N$ are non-D-finite and $(N(n))$ is not P-recursive.
\end{lemma}

\begin{proof}
The Taylor remainder divided by $z^{K+1}$ extends continuously to $[0,1]$ and is bounded there. Substituting $z=e^{-t}$ and integrating proves \eqref{mellin:remainder}.
On fixed-radius disks centered at large positive real $s$, Stirling's formula bounds $\Gamma(2\Re s)/|\Gamma(2s)|$ uniformly. Cauchy's estimates then give the derivative assertion.

At integer orders write
\[
 r_K(n)=N(n)-\sum_{k=1}^K a_kk^{-2n},\qquad
 E_K(z)=\sum_{n\ge1}r_K(n)z^n.
\]
The bound $|r_K(n)|\le C_K(K+1)^{-2n}$ shows that $E_K$ is holomorphic in $|z|<(K+1)^2$, and summing the finite geometric series gives
\begin{equation}\label{mellin:local-partial-fractions}
 C(z)=\sum_{k=1}^K\frac{a_kz}{k^2-z}+E_K(z).
\end{equation}
These continuations agree near zero and hence on their common domains. Increasing $K$ proves the pole and residue assertions throughout the plane. Infinitely many poles exclude D-finiteness: a solution of a linear differential equation with rational coefficients can have finite-plane singularities only at the finitely many singular points of that equation. The D-finite/P-recursive equivalence \cite[Theorem 1.5]{Stanley} proves the recurrence assertion.

For the continuous-variable conclusion, suppose $\sum_{j=0}^r p_j(s)N^{(j)}(s)=0$ with $p_j\in\CC[s]$ and put
\[
 Q_k(s)=\sum_{j=0}^r p_j(s)(-2\log k)^j,\qquad d=\max_j\deg p_j.
\]
Successively compare the terms $k^{-2s}$ with $a_k\ne0$. After the earlier terms have been eliminated, \eqref{mellin:remainder} and its derivative bounds give
\[
 Q_k(s)=O\!\left(s^d\left(\frac{k}{k+1}\right)^{2s}\right)
 \qquad(s\to+\infty).
\]
Thus $Q_k$ is the zero polynomial. For each power of $s$, its coefficient in $\sum_jp_j(s)X^j$ has the infinitely many roots $-2\log k$ for which $a_k\ne0$. Every $p_j$ therefore vanishes. Clearing denominators handles rational coefficients.
\end{proof}

No infinite Dirichlet-series representation is assumed in this lemma. If $W$ is a polynomial, conversely, $N$ is a finite sum of exponentials and $C$ is rational. Thus $N$ and $C$ are D-finite, and $(N(n))_{n\ge1}$ is P-recursive, if and only if $W$ is a polynomial.

\section{Generating functions for Mellin moments}\label{sec:moments}

For $x>0$, put $f(x)=R(e^{-\pi\sqrt{x}})$ and consider
\begin{equation}\label{intro:moments}
 M(s)=\int_0^\infty x^{s-1}f(x)\,dx,\qquad
 G(s)=\int_0^\infty x^{s-1}\log(1+f(x)/\phi)\,dx
 \quad(\Re s>0).
\end{equation}
The substitution $q=e^{-\pi\sqrt{x}}$ gives
\[
 \int_0^1 R(q)\left(\log\frac1q\right)^{w-1}\frac{dq}{q}
 =\frac{\pi^w}{2}M(w/2)\qquad(\Re w>0),
\]
with the analogous identity for $G$.
After meromorphic continuation, define
\begin{equation}\label{intro:normalized}
 B_+(s)=\frac{(\pi/5)^{2s}M(s)}{2\Gamma(2s)},\quad
 B_-(s)=-\frac{(4\pi/5)^{2s}M(-s)}{2\Gamma(2s)},\quad
 B_G(s)=\frac{(\pi/5)^{2s}G(s)}{2\Gamma(2s)}.
\end{equation}
Here and below, $M(-s)$ denotes the continuation, not the divergent defining integral at a negative real order. Write
\[
 C_\nu(z)=\sum_{n\ge1}B_\nu(n)z^n\qquad(\nu=+,-,G).
\]

\begin{theorem}[Poles and residues]\label{thm:main}
The functions $M$ and $G$ continue meromorphically to $\CC$, with their only poles at zero, simple and of residues $\phi^{-1}$ and $\log(\sqrt5/\phi)$, respectively. The three functions in \eqref{intro:normalized} are entire. Each $C_\nu$ has radius of convergence $1$ and a meromorphic continuation to $\CC$ with the following simple poles and residues:
\begin{center}
\begin{tabular}{@{}lll@{}}
\toprule
Function & Pole locations & Residue \\
\midrule
$C_+$ & $(5m+1)^2$, $m\ge0$, $m\notin\{3,8,13,23\}$ & $-(5m+1)^2P_m$ \\
$C_-$ & $k^2$, $k\ge1$ & $-k^2d_k$ \\
$C_G$ & $k^2$, $k\ge1$ & $-k^2c_k$ \\
\bottomrule
\end{tabular}
\end{center}
Every $d_k$ is a quadratic irrational algebraic integer, and every $c_k$ belongs to $\QQ(\sqrt5)\setminus\QQ$. The functions $B_\nu$ and $C_\nu$ are non-D-finite, and the sequences $B_\nu(n)$, $M(n)$, $M(-n)$ and $G(n)$ are non-P-recursive.
\end{theorem}

The four omitted values $(5m+1)^2$ are $256$, $1681$, $4356$
and $13456$. The nonvanishing results of Section~\ref{sec:coefficients}
determine the remaining poles.

Ramanujan's reciprocity identity \cite{WatsonIX,AndrewsBerndt} reads
\begin{equation}\label{moments:reciprocity}
 (\phi+f(x))(\phi+f(16/x))=\phi\sqrt5.
\end{equation}
Together with the product for $R$, it gives the two endpoint expansions
\begin{equation}\label{moments:endpoints}
\begin{aligned}
 f(x)&=e^{-\pi\sqrt{x}/5}\bigl(1+O(e^{-\pi\sqrt{x}})\bigr)
       &&(x\longrightarrow\infty),\\
 f(x)&=\phi^{-1}-\frac{\sqrt5}{\phi}
       e^{-4\pi/(5\sqrt{x})}+O\bigl(e^{-8\pi/(5\sqrt{x})}\bigr)
       &&(x\longrightarrow0^+).
\end{aligned}
\end{equation}
Consequently the regularized representation
\begin{equation}\label{moments:regularization}
 M(s)=\frac{4^s}{\phi s}
      +\int_0^4x^{s-1}\bigl(f(x)-\phi^{-1}\bigr)\,dx
      +\int_4^\infty x^{s-1}f(x)\,dx
\end{equation}
continues $M$ meromorphically to the plane, with only a simple pole at zero,
of residue $\phi^{-1}$: both integrals are entire in $s$.
The identical argument for $g(x)=\log(1+f(x)/\phi)$, whose limit at zero is
$h_0=\log(\sqrt5/\phi)$, proves the corresponding assertion for $G$.

For $\Re s>0$, substitute $x\mapsto16/x$ in
\eqref{moments:regularization} at $-s$. Using
$\phi^{-1}-f(16/x)=\sqrt5f(x)/(\phi+f(x))$ gives
\begin{equation}\label{moments:reflected}
 -16^sM(-s)=\int_0^\infty x^{s-1}
                 \frac{\sqrt5f(x)}{\phi+f(x)}\,dx.
\end{equation}
The substitution $t=\pi\sqrt{x}/5$ now identifies the normalized transforms:
\begin{equation}\label{moments:kernels}
 \begin{pmatrix}B_+(s)\\ B_-(s)\\ B_G(s)\end{pmatrix}
 =\frac1{\Gamma(2s)}\int_0^\infty t^{2s-1}
       \begin{pmatrix}H(e^{-t})\\ D(e^{-t})\\ V(e^{-t})\end{pmatrix}\,dt
       \qquad(\Re s>0).
\end{equation}
The only poles of $M$ and $G$ are canceled by $1/\Gamma(2s)$, so all three
normalized functions are entire, and
\begin{equation}\label{moments:zero}
 B_+(0)=B_-(0)=\phi^{-1},\qquad B_G(0)=h_0.
\end{equation}
The kernels are analytic in the unit disk and continuous on $[0,1]$, with
$H(1)=D(1)=\phi^{-1}$ and $V(1)=h_0$.

\begin{proof}[Proof of Theorem~\ref{thm:main}]
The continuation assertions follow from \eqref{moments:regularization}
and its logarithmic analogue. Apply Lemma~\ref{mellin:transfer} to the three
kernels in \eqref{moments:kernels}. The coefficient results
\ref{coeff:arithmetic} and \ref{coeff:support} give their exact supports;
the lemma therefore gives precisely the poles and residues in the theorem,
and the non-D-finiteness of the three generating functions and the three
normalized transforms. The coefficients of $z$ in $H,D,V$ are respectively
$1$, $\sqrt5/\phi$, and $\phi^{-1}$. Thus each $C_\nu$ has a pole at
$z=1$ and no singularity closer to zero, so its radius of convergence is $1$.

The lemma also excludes polynomial-coefficient recurrences for the three
normalized moment sequences. To transfer this conclusion, put
$J(n)=2\Gamma(2n)(5/\pi)^{2n}$. For each integer $j\geq0$,
\[
 \frac{J(n+j)}{J(n)}
   =\left(\frac5\pi\right)^{2j}
       \prod_{\ell=0}^{2j-1}(2n+\ell),
\]
where an empty product is $1$. Thus a recurrence for $M(n)=J(n)B_+(n)$
or $G(n)=J(n)B_G(n)$ would give one for its normalized sequence.
The same argument applies to
$M(-n)=-2\Gamma(2n)(5/(4\pi))^{2n}B_-(n)$.
This proves all the recurrence assertions.
\end{proof}

\begin{corollary}[Signs of the residues]
\label{moments:residue-signs}
Index the poles of $C_G$ by the positive integer $k$, so that the
corresponding pole is $k^2$.
In each progression $k\equiv b\pmod M$, either all residues are positive
or both residue signs have positive natural density, with the alternative
determined by Theorem~\ref{thm:congruence} after reversing signs.
On $v_5(k)=a$, the relative density of positive residues is $\delta_a^-$.
Their first pole is
\[
 \lambda_a=(5^am_a)^2,\qquad
 \lim_{a\to\infty}5^{-a}\log\log\lambda_a
       =\frac{5\log\phi}{12e^\gamma}.
\]
Here natural density refers to sets of positive integers $k$.
\end{corollary}
\begin{proof}
The residue at $k^2$ is $-k^2c_k$. Apply
Theorems~\ref{thm:congruence} and \ref{thm:rarity}; multiplication by
$5^a$ and squaring do not change the stated double-logarithmic limit.
\end{proof}

The finite expansions in Lemma~\ref{mellin:transfer} give, for each fixed
$K$ and real $s\longrightarrow+\infty$,
\begin{equation}\label{moments:asymptotics}
\begin{aligned}
 B_+(s)&=\sum_{m=0}^{K}\frac{P_m}{(5m+1)^{2s}}
                +O_K\bigl((5K+6)^{-2s}\bigr),\\
 B_-(s)&=\sum_{k=1}^{K}\frac{d_k}{k^{2s}}
                +O_K\bigl((K+1)^{-2s}\bigr),\\
 B_G(s)&=\sum_{k=1}^{K}\frac{c_k}{k^{2s}}
                +O_K\bigl((K+1)^{-2s}\bigr).
\end{aligned}
\end{equation}
For the first line, bound the remainder of $F$ after degree $K$ and use
$H(z)=zF(z^5)$. In particular,
\[
 M(s)\sim2\Gamma(2s)(5/\pi)^{2s},\qquad
 M(-s)\sim-\frac{2\sqrt5}{\phi}
                  \Gamma(2s)(5/(4\pi))^{2s},\qquad
 B_G(s)\longrightarrow\phi^{-1}.
\]
The first line of \eqref{moments:asymptotics} is a finite asymptotic
expansion; the full formal Dirichlet series diverges at every fixed order.

\begin{proposition}\label{moments:logseries}
The logarithmic transform has the absolutely convergent representation
\begin{equation}\label{moments:dirichlet}
 B_G(s)=\sum_{k\geq1}\frac{c_k}{k^{2s}}
             \qquad(\Re s>1/2).
\end{equation}
Its generating function has the partial-fraction expansion
\begin{equation}\label{moments:partialfractions}
 C_G(z)=\sum_{k\geq1}\frac{c_kz}{k^2-z},
        \qquad z\in\mathbb C\setminus\{1,4,9,\ldots\},
\end{equation}
with absolute and locally uniform convergence on the stated domain.
\end{proposition}

\begin{proof}
The bound $c_k=O(1+\log k)$ justifies termwise Mellin integration in
\eqref{moments:kernels} for $\Re s>1/2$, proving
\eqref{moments:dirichlet}. For $|z|<1$, absolute summation of the geometric
series gives \eqref{moments:partialfractions}; on every compact set avoiding
the positive squares its tail is bounded by a constant times
$\sum_k|c_k|/k^2$. The resulting meromorphic function is the continuation
of $C_G$.
\end{proof}

An incomplete-gamma expansion retains convergence for the ordinary
moments despite the growth of $P_m$.
Here $\Gamma(u,y)=\int_y^\infty t^{u-1}e^{-t}\,dt$ for $y>0$.
\begin{proposition}\label{moments:incompletegamma}
For $a,b>0$ with $ab=4\pi^2/25$,
\begin{equation}\label{moments:gammaformula}
\begin{aligned}
 M(s)={}&\frac{(5a/\pi)^{2s}}{\phi s}
  +2\left(\frac5\pi\right)^{2s}
       \sum_{m\geq0}\frac{P_m\Gamma(2s,a(5m+1))}{(5m+1)^{2s}}\\
 &-2\left(\frac{4\pi}5\right)^{2s}
       \sum_{k\geq1}d_k k^{2s}\Gamma(-2s,bk).
\end{aligned}
\end{equation}
Both series converge locally uniformly to entire functions of $s$.
\end{proposition}

\begin{proof}
In $M(s)=2\pi^{-2s}\int_0^\infty t^{2s-1}R(e^{-t})\,dt$,
split at $5a$. On the upper interval expand
$R(e^{-t})=\sum_{m\geq0}P_me^{-(5m+1)t/5}$; on the lower interval use
\[
 R(e^{-t})=\phi^{-1}-D(e^{-4\pi^2/(5t)}).
\]
Termwise integration gives \eqref{moments:gammaformula} for $\Re s>0$.
Indeed, Cauchy bounds give $|P_m|\leq C_r r^{-m}$ and
$|d_k|\leq C_r r^{-k}$ for every $r<1$. Choose $r$ sufficiently close
to $1$ that $-\log r<\min(5a,b)$. The incomplete gamma factors contribute
the decays $e^{-a(5m+1)}$ and $e^{-bk}$, up to factors bounded by fixed
powers of the indices on each compact set of orders. These dominate the
Cauchy bounds and prove both the interchanges and local uniform
convergence. For a positive lower endpoint, $\Gamma(u,y)$ is entire in
$u$; hence analytic continuation proves the identity for all $s$.
\end{proof}

\begin{corollary}\label{moments:nonDfiniteM}
The meromorphic function $M(s)$ is not D-finite over $\mathbb C(s)$.
\end{corollary}

\begin{proof}
This needs a direct argument, since gamma normalization does not preserve
D-finiteness. Put $J(s)=2\Gamma(2s)(5/\pi)^{2s}$. The continued fraction
gives
$0<e^{-\pi\sqrt{x}/5}-f(x)<e^{-6\pi\sqrt{x}/5}$.
For real $s>1$, differentiation under the integral and
$|\log x|^j\leq j!(x+x^{-1})$ therefore give
\[
 |(J-M)^{(j)}(s)|
 \leq j!\{6^{-2s-2}J(s+1)+6^{-2s+2}J(s-1)\}.
\]
After division by $J(s)$ the right side is $O_j(s^2 6^{-2s})$.
Differentiated Stirling asymptotics \cite[\S\S5.11,5.15]{NIST} give
$(\log J)'(s)=2\log s+O(1)$ and
$(\log J)^{(r)}(s)=O_r(s^{1-r})$ for $r\geq2$. Thus
$M^{(j)}(s)/M(s)\sim(2\log s)^j$ for every fixed $j$.
In a putative relation $\sum_jp_j(s)M^{(j)}(s)=0$, choose the maximal
degree of a nonzero $p_j$, and among coefficients of that degree choose
the largest $j$. Division by $M(s)$ and the corresponding powers of
$s$ and $\log s$ leaves a single nonzero limiting constant, a
contradiction. Clearing denominators handles rational coefficients.
\end{proof}

\section{Integral evaluations}\label{sec:logarithmic}

Classical Dirichlet-series and modular-form identities give the
following integral evaluations.

\subsection{The logarithmic Mellin transform}
\begin{proposition}\label{log:transform}
As identities of meromorphic functions,
\begin{equation}\label{log:evaluation}
\begin{aligned}
G(s)=\Gamma(2s)\left(\frac5\pi\right)^{2s}
\bigl\{&\sqrt5\,\zeta(2s)L(2s+1,\chi)
       -5^{-2s}\zeta(2s+1)L(2s,\chi)\\
       &-(1-5^{-4s})\zeta(2s)\zeta(2s+1)\bigr\},
\end{aligned}
\end{equation}
and
\begin{equation}\label{log:reflection}
G(s)=-16^sG(-s).
\end{equation}
The only pole of $G$ is at zero, simple and of residue
$h_0=\log(\sqrt5/\phi)$.
\end{proposition}
\begin{proof}
Integrate the four terms of Ramanujan's Lambert expansion
\eqref{coeff:lambert} after $z=e^{-\pi\sqrt{x}/5}$. Absolute
integration for $\Re s>1/2$ follows from
\[
\int_0^\infty x^{s-1}e^{-\pi N\sqrt{x}/5}\,dx
 =2\Gamma(2s)\left(\frac5{\pi N}\right)^{2s}.
\]
The resulting divisor-series products give \eqref{log:evaluation}.
For $g(x)=\log(1+f(x)/\phi)$, reciprocity gives
$g(x)+g(16/x)=h_0$. Splitting at $4$ consequently yields
\begin{equation}\label{log:split}
4^{-s}G(s)=\frac{h_0}{s}
 +2\int_1^\infty
   (u^{2s-1}-u^{-2s-1})g(4u^2)\,du.
\end{equation}
The integral is entire and odd, because
$g(x)=O(e^{-\pi\sqrt{x}/5})$ at infinity. This proves the pole
statement, reflection, and meromorphic continuation of the evaluation.
\end{proof}

The integer and half-integer formulas, their invertible special-value
relation, and the value at $s=1/2$ are collected in
Appendix~\ref{sec:restored-special}. They follow by substituting standard
even zeta and Dirichlet $L$-values in \eqref{log:evaluation}.

\subsection{A quadratic integral}\label{sec:quadratic}
Define the holomorphic logarithm, normalized by $\mathcal U(0)=0$,
\begin{equation}\label{quad:kernel}
 \mathcal U(z)=\log\frac{1+H(z)/\phi}{1-\phi H(z)}.
\end{equation}
Equation~\eqref{coeff:products} ensures that both factors are nonzero
and gives
\begin{equation}\label{quad:coefficients}
 \mathcal U(z)=\sqrt5\sum_{n\ge1}A_nz^n,
 \qquad A_n=\sum_{d\mid n}\frac{\chi(d)}d.
\end{equation}

\begin{proposition}\label{quad:transform}
The integral
\begin{equation}\label{quad:integral}
 \mathcal I(s)=\frac1{2\pi}\int_0^1
       \left(\log\frac1r\right)^{s-1}
       \int_0^{2\pi}|\mathcal U(re^{i\theta})|^2\,d\theta\,
       \frac{dr}{r}
\end{equation}
converges absolutely exactly for $\Re s>1$, where
\begin{equation}\label{quad:evaluation}
 \mathcal I(s)=\frac{5\Gamma(s)}{2^s}
       \frac{1-5^{-s-2}}{1-5^{-2s-2}}
       \frac{\zeta(s)\zeta(s+2)L(s+1,\chi)^2}{\zeta(2s+2)}.
\end{equation}
Its meromorphic continuation has just one pole in $\Re s>0$:
\begin{equation}\label{quad:residue}
 \Res_{s=1}\mathcal I(s)=\frac{372}{325}\zeta(3).
\end{equation}
\end{proposition}
\begin{proof}
Parseval and Mellin integration, justified by $A_n=O(\log(2n))$, give
\begin{equation}\label{quad:dirichlet}
 \mathcal I(s)=\frac{5\Gamma(s)}{2^s}
                 \sum_{n\ge1}\frac{A_n^2}{n^s}.
\end{equation}
The character-weighted divisor-product identity recorded by Murty
\cite[Theorem~5.2 and its specialization, p.~83]{Murty}, with both
divisor exponents $-1$ and both characters $\chi$, states that
\begin{equation}\label{quad:euler}
 \sum_{n\ge1}\frac{A_n^2}{n^s}
 =\frac{\zeta(s)L(s+1,\chi)^2L(s+2,\chi^2)}
        {L(2s+2,\chi^2)}.
\end{equation}
Since $L(w,\chi^2)=(1-5^{-w})\zeta(w)$, this proves the evaluation.
Its denominator has no zeros for $\Re s>0$; the only pole there is
at $1$, and $L(2,\chi)=4\pi^2/(25\sqrt5)$ gives the residue.
The positive residue and monotone convergence imply
$\sum A_n^2/n=\infty$. Tonelli's theorem proves divergence of
\eqref{quad:integral} at $s=1$. On $e^{-1}<r<1$ the radial weight
for real $s\le1$ is at least $1$, whereas the complementary integral
is finite. This proves the exact absolute-convergence domain.
\end{proof}

\begin{corollary}\label{quad:odd-squares}
For $m\ge1$, with $B_j$ denoting the usual Bernoulli numbers,
\begin{equation}\label{quad:even-orders}
 \mathcal I(2m)=C_mL(2m+1,\chi)^2,\qquad
 C_m=-\frac{5(4m+2)!B_{2m}B_{2m+2}}
 {m\,2^{2m+2}(2m+2)!B_{4m+2}}
       \frac{1-5^{-2m-2}}{1-5^{-4m-2}}\in\QQ_{>0}.
\end{equation}
In particular,
\begin{equation}\label{quad:cubic-square}
 \mathcal I(2)=\frac{1625}{744}L(3,\chi)^2.
\end{equation}
\end{corollary}
\begin{proof}
Substitute the standard Bernoulli formulas for the three even zeta
values in \eqref{quad:evaluation}; their powers of $\pi$ cancel.
\end{proof}

\subsection{A Mellin transform involving a cusp form}
Put $u(t)=\mathcal U(e^{-t})$, and, for $q=e^{2\pi i\tau}$, define
\begin{equation}\label{cusp:definitions}
 \Delta_5(q)=\eta(\tau)^4\eta(5\tau)^4
   =\sum_{n\ge1}\tau_5(n)q^n,\qquad
 L_5(s)=\sum_{n\ge1}\frac{\tau_5(n)}{n^s}\quad(\Re s>4).
\end{equation}
Here $\eta(\tau)=q^{1/24}(q;q)_\infty$, and $\Delta_5$ is the
classical weight-four cusp form of level $5$.

\begin{proposition}\label{cusp:square}
For $\Re s>4$,
\begin{equation}\label{cusp:evaluation}
 \mathcal Q(s):=\int_0^\infty t^{s-1}[u'(t)]^2\,dt
 =\frac{5\Gamma(s)}{13}
 \{(1-5^{-s})\zeta(s)\zeta(s-3)-L_5(s)\}.
\end{equation}
This is the exact domain of absolute convergence. Its meromorphic
continuation has only the simple pole at $s=4$, of residue $16\pi^4/625$.
\end{proposition}
\begin{proof}
Ramanujan's derivative identity \cite{BerndtNotebooksIII} and
\eqref{coeff:ramanujan} give the classical eta quotients
\begin{equation}\label{cusp:eta-pair}
 \mathcal B(q)=5q\frac{d}{dq}\log R(q)
       =\frac{\eta(\tau)^5}{\eta(5\tau)},\qquad
 \mathcal A(q)=\frac{q\mathcal U'(q)}{\sqrt5}
       =\frac{\eta(5\tau)^5}{\eta(\tau)}.
\end{equation}
For the second identity use $qH'/H=\mathcal B(q^5)$ and
$H/(1-H-H^2)=q(q^{25};q^{25})_\infty/(q;q)_\infty$.
Thus $u'(t)=-\sqrt5\mathcal A(e^{-t})$ and
$\mathcal A\mathcal B=\Delta_5$.
Write $E_4(\tau)=1+240\sum_{n\ge1}\sigma_3(n)q^n$, where
$\sigma_3(n)=\sum_{d\mid n}d^3$.
Cooper--Toh \cite[Theorem~2.6]{CooperToh}, with
$X=\mathcal A/\mathcal B$ and $Z=\mathcal B$, record Ramanujan's formulas
\[
 E_4(\tau)=Z^2(1+250X+3125X^2),\qquad
 E_4(5\tau)=Z^2(1+10X+5X^2).
\]
Their difference is exactly
\begin{equation}\label{cusp:modular-square}
 3120\mathcal A^2=E_4(\tau)-E_4(5\tau)-240\Delta_5.
\end{equation}
The coefficients of $\mathcal A$ are $nA_n=O(n\log(2n))$, so those
of $\mathcal A^2$ and $\Delta_5$ are $O(n^3\log^2(2n))$.
Termwise Mellin integration for $\Re s>4$, using
$\sum\sigma_3(n)n^{-s}=\zeta(s)\zeta(s-3)$, proves the formula.

With $c=4\pi^2/5$, reciprocity gives
\begin{equation}\label{cusp:reciprocity}
 u(t)=-\log\phi-\log R(e^{-c/t}),\qquad
 u'(t)=-\frac{c}{5t^2}\mathcal B(e^{-c/t}).
\end{equation}
Hence $u'(t)^2=16\pi^4/(625t^4)+O(t^{-4}e^{-c/t})$ at zero and
$u'(t)=O(e^{-t})$ at infinity. Subtracting the one singular term on
$(0,1)$ leaves an entire Mellin integral and the term
$16\pi^4/(625(s-4))$. These estimates also prove the exact convergence domain.
\end{proof}

\begin{corollary}\label{cusp:bilinear}
For every $s\in\CC$, the absolutely convergent integral
\begin{equation}\label{cusp:bilinear-integral}
 \mathcal J(s)=\int_0^1\left(\log\frac1q\right)^{s-1}
 \left(q\frac{d}{dq}\log R(q)\right)
 \left(q\mathcal U'(q)\right)\frac{dq}{q}
 =\frac{\Gamma(s)}{\sqrt5}L_5(s)
\end{equation}
is entire and satisfies
\begin{equation}\label{cusp:reflection}
 \mathcal J(s)=\left(\frac{4\pi^2}{5}\right)^{s-2}\mathcal J(4-s).
\end{equation}
The product on the right of \eqref{cusp:bilinear-integral} is understood
by its entire continuation.
\end{corollary}
\begin{proof}
The two derivative factors multiply to $\Delta_5(q)/\sqrt5$.
For $c=4\pi^2/5$, the eta transformation gives
\[
 \Delta_5(e^{-t})=c^2t^{-4}\Delta_5(e^{-c/t}).
\]
The kernel decays exponentially at both endpoints, after the indicated
power at zero, so its Mellin transform is entire.
Termwise integration first for $\Re s>4$ gives
\eqref{cusp:bilinear-integral}; the substitution $t=c/v$ gives
\eqref{cusp:reflection}. These formulas give the Mellin transform and functional equation
of the cusp form $\Delta_5$.
\end{proof}

\appendix

\section{A functional relation and bounds in vertical strips}\label{sec:restored-analysis}

We record a functional relation for $M$ and bounds in vertical strips.
\begin{proposition}\label{restore:reflection}
As an identity of meromorphic functions,
\begin{equation}\label{restore:bilinear}
 M(s)+16^sM(-s)
   =-\frac1\phi\int_0^\infty x^{s-1}f(x)f(16/x)\,dx.
\end{equation}
The integral on the right is entire in $s$, and the left side has a
removable value at zero. For real $s>0$, the function
\[
 A(s)=-\frac{16^sM(-s)}{M(s)}
\]
increases strictly from $1$ to $\sqrt5/\phi$. Consequently
\begin{equation}\label{restore:sharp-bounds}
 -\frac{\sqrt5}{\phi}16^{-s}M(s)
    <M(-s)<-16^{-s}M(s)<0\qquad(s>0),
\end{equation}
and both endpoint constants are optimal. Moreover, $M$ is strictly
log-convex on the positive real axis.
\end{proposition}

\begin{proof}
The endpoint estimates \eqref{moments:endpoints} make the integral in
\eqref{restore:bilinear} entire. Substitute $x\mapsto16/x$ in the
regularized representation \eqref{moments:regularization} for
$16^sM(-s)$, and add the representation for $M(s)$. The constant terms
cancel, while reciprocity gives
\[
 f(x)f(16/x)+\phi\bigl(f(x)+f(16/x)\bigr)=1.
\]
This proves \eqref{restore:bilinear}, including removability at zero.

Ramanujan's derivative identity \cite{BerndtNotebooksIII}, also obtainable
by differentiating \cite[Eq.~(3)]{BagisGlasser}, is
\[
 q\frac{d}{dq}\log R(q)
    =\frac{(q;q)_\infty^5}{5(q^5;q^5)_\infty}>0
       \qquad(0<q<1).
\]
Thus $f$ decreases strictly from $\phi^{-1}$ to zero. For $s>0$, define
the probability measure
\[
 d\mu_s(x)=\frac{x^{s-1}f(x)}{M(s)}\,dx,
 \qquad r(x)=\frac{\sqrt5}{\phi+f(x)}.
\]
By \eqref{moments:reflected}, $A(s)=\int r\,d\mu_s$.
The function $r$ increases strictly from $1$ to $\sqrt5/\phi$, so
the inequalities in \eqref{restore:sharp-bounds} are strict.
Differentiation under the integral is justified by the endpoint bounds,
and gives
\[
 A'(s)=\frac12\int_0^\infty\!\int_0^\infty
       (r(x)-r(y))(\log x-\log y)\,d\mu_s(x)\,d\mu_s(y)>0.
\]
The residue of $M$ at zero yields $A(s)\to1$ as $s\to0^+$.
For the other endpoint, the continued fraction gives
\[
 0<e^{-\pi\sqrt{x}/5}-f(x)<e^{-6\pi\sqrt{x}/5},
 \qquad
 0<1-\frac{M(s)}{2\Gamma(2s)(5/\pi)^{2s}}<6^{-2s}.
\]
For any fixed $L>0$,
\[
 \mu_s((0,L))\leq\frac{L^s}{s\phi M(s)}\longrightarrow0
       \qquad(s\longrightarrow\infty).
\]
Since $r$ is bounded and tends to $\sqrt5/\phi$ at infinity, its
$\mu_s$-mean tends to that constant. This proves optimality.
Finally, the same differentiation gives
\[
 (\log M)''(s)
     =\int(\log x)^2\,d\mu_s(x)
       -\left(\int\log x\,d\mu_s(x)\right)^2>0,
\]
because $\log x$ is not constant almost everywhere for this measure.
\end{proof}

\begin{proposition}\label{restore:strip-decay}
For every compact real interval $I$ and every $0<\varepsilon<\pi$,
\begin{equation}\label{restore:vertical-bound}
 M(\sigma+it)=O_{I,\varepsilon}
                   \bigl(e^{-(\pi-\varepsilon)|t|}\bigr)
       \qquad(\sigma\in I,\ |t|\geq1).
\end{equation}
\end{proposition}

\begin{proof}
Using the principal square root and the definition of $R(e^{-t})$
in the introduction, $f(z)=e^{-\pi\sqrt z/5}F(e^{-\pi\sqrt z})$
is holomorphic in $|\arg z|<\pi$.
The product and reciprocity extend the endpoint bounds uniformly to
each closed subsector: the respective errors are bounded by
$Ce^{-c\sqrt{|z|}}$ at infinity and $Ce^{-c/\sqrt{|z|}}$ at zero,
with positive constants depending on the subsector. Cauchy's estimate
in a slightly larger subsector therefore gives
\begin{equation}\label{restore:sector-derivative}
 |f'(z)|\leq\frac{C}{|z|}
 \begin{cases}
  e^{-c/\sqrt{|z|}},&0<|z|\leq1,\\
  e^{-c\sqrt{|z|}},&|z|\geq1.
 \end{cases}
\end{equation}
Integration by parts first for $\Re s>0$, and then analytic
continuation, gives the identity of entire functions
\[
 sM(s)=-\int_0^\infty x^sf'(x)\,dx.
\]
Rotate the integration ray to
$\theta=\operatorname{sgn}(t)(\pi-\varepsilon)$.
The integrals over both circular arcs tend to zero by \eqref{restore:sector-derivative}.
Writing $z=e^{i\theta}u$ gives the factor $e^{i\theta(s+1)}$, of
modulus $e^{-(\pi-\varepsilon)|t|}$, and
\[
 \sup_{\sigma\in I}
     \int_0^\infty u^\sigma|f'(e^{i\theta}u)|\,du<\infty.
\]
Division by $|s|\geq1$ proves \eqref{restore:vertical-bound}.
\end{proof}

\section{Special values of the logarithmic Mellin transform}\label{sec:restored-special}
We give the integer and half-integer values of $G$, and an expression
in terms of completed $L$-functions, using Proposition~\ref{log:transform}.

For the special values, use the Bernoulli convention
\[
\frac{te^{ut}}{e^t-1}=\sum_{j\ge0}B_j(u)\frac{t^j}{j!},
\qquad B_j=B_j(0),\qquad
B_{j,\chi}=5^{j-1}\sum_{a=1}^5\chi(a)B_j(a/5).
\]

\begin{corollary}\label{log:special-values}
For each integer $n\ge1$,
\begin{equation}\label{log:integer}
\begin{aligned}
G(n)=\frac{(-1)^{n+1}}{4n}\biggl\{
 &100^nB_{2n}\bigl[\sqrt5\,L(2n+1,\chi)
                -(1-5^{-4n})\zeta(2n+1)\bigr]\\
 &-\sqrt5\left(\frac4{25}\right)^n
       B_{2n,\chi}\zeta(2n+1)\biggr\},
\end{aligned}
\end{equation}
and
\begin{equation}\label{log:half-integer}
\begin{aligned}
\frac{G(n+1/2)}{\pi}
 =\frac{(-1)^n2^{2n+1}}{(2n+1)(2n+2)}\bigl\{
 &[B_{2n+2,\chi}-(5^{2n+1}-5^{-2n-1})B_{2n+2}]
     \zeta(2n+1)\\
 &-B_{2n+2}L(2n+1,\chi)\bigr\}.
\end{aligned}
\end{equation}
These formulas give an invertible linear relation over
$\mathbb Q(\sqrt5)$ between
$(G(n),G(n+1/2)/\pi)$ and
$(\zeta(2n+1),L(2n+1,\chi))$.
\end{corollary}

\begin{proof}
The standard even-value formulas
\cite[\S24.8]{NIST}, \cite[Chapter~4]{Washington} are
\[
\zeta(2n)=\frac{(-1)^{n+1}(2\pi)^{2n}B_{2n}}{2(2n)!},
\qquad
L(2n,\chi)=\frac{(-1)^{n+1}\sqrt5}{2(2n)!}
             \left(\frac{2\pi}{5}\right)^{2n}B_{2n,\chi}.
\]
Substitution in \eqref{log:evaluation} proves
\eqref{log:integer} and \eqref{log:half-integer}.
To check invertibility, write their coefficient matrix, with columns
ordered as $\zeta(2n+1)$ and $L(2n+1,\chi)$, as
$\left(\begin{smallmatrix}a&b\\c&d\end{smallmatrix}\right)$.
Here $b\in\sqrt5\,\mathbb Q$, $c,d\in\mathbb Q$, and $d\ne0$.
The rational part of $a$ in the basis $1,\sqrt5$ is
$(-1)^n100^nB_{2n}(1-5^{-4n})/(4n)\ne0$.
Thus $ad-bc$ has nonzero rational part and cannot vanish.
\end{proof}

\begin{proposition}\label{restore:half-order}
Let $\mathfrak A$ be the Glaisher--Kinkelin constant, normalized by
$\log\mathfrak A=1/12-\zeta'(-1)$. Then
\begin{equation}\label{restore:half-value}
 G\!\left(\tfrac12\right)
 =\pi\left\{\frac{48}{5}\log\mathfrak A+2L'(-1,\chi)
       -\frac45-\frac{13}{15}\log5
       -\frac{\log\phi}{3\sqrt5}\right\}.
\end{equation}
Equivalently,
\[
 \int_0^1\log\left(1+\frac{R(q)}\phi\right)\frac{dq}{q}
   =\frac\pi2G\!\left(\tfrac12\right).
\]
\end{proposition}

\begin{proof}
Put $\Phi(w)=\sqrt5L(w+1,\chi)-(1-5^{-2w})\zeta(w+1)$.
The even-value formulas give $\Phi(1)=0$, so
$\lim_{w\to1}\zeta(w)\Phi(w)=\Phi'(1)$.
The elementary finite Fourier evaluation
\[
 L(1,\chi)=-\frac1{\sqrt5}
     \sum_{a=1}^4\chi(a)\log\bigl(2\sin(\pi a/5)\bigr)
     =\frac{2\log\phi}{\sqrt5}
\]
follows from the logarithmic sine series, with an Abel factor if
necessary, and $\sin(2\pi/5)/\sin(\pi/5)=\phi$.
Taking the limit in \eqref{log:evaluation} therefore yields
\begin{equation}\label{restore:half-derivatives}
 G\!\left(\tfrac12\right)
 =\frac{5\sqrt5}{\pi}L'(2,\chi)-\frac{24}{5\pi}\zeta'(2)
       -\frac\pi{15}\log5-\frac\pi{3\sqrt5}\log\phi.
\end{equation}
Logarithmic differentiation of the functional equations at $w=2$
\cite[\S\S25.4,25.15]{NIST} gives
\[
 \frac{\zeta'(2)}{\zeta(2)}
       =\gamma+\log(2\pi)-12\log\mathfrak A,
 \qquad
 \frac{L'(2,\chi)}{L(2,\chi)}
       =\frac52L'(-1,\chi)+\gamma-1+\log(2\pi/5).
\]
Here $L(-1,\chi)=-2/5$ and $\gamma$ is Euler's constant.
Since $(5\sqrt5/\pi)L(2,\chi)=24\zeta(2)/(5\pi)=4\pi/5$,
substitution in \eqref{restore:half-derivatives} proves
\eqref{restore:half-value}. The integral identity follows by
$q=e^{-\pi\sqrt{x}}$.
\end{proof}

Define the completed functions
\[
 \Lambda_\zeta(u)=\pi^{-u/2}\Gamma(u/2)\zeta(u),\qquad
 \Lambda_\chi(u)=(5/\pi)^{u/2}\Gamma(u/2)L(u,\chi).
\]
Both satisfy $\Lambda(u)=\Lambda(1-u)$; $\Lambda_\chi$ is entire,
whereas $\Lambda_\zeta$ has simple poles at $0,1$.
Gamma duplication in \eqref{log:evaluation}, with $w=2s$, gives
\begin{equation}\label{restore:completed}
\begin{aligned}
 2^{-w}G(w/2)=\frac12\bigl\{
 &5^{w/2}\Lambda_\zeta(w)\Lambda_\chi(w+1)
  -5^{-w/2}\Lambda_\chi(w)\Lambda_\zeta(w+1)\\
 &-(5^w-5^{-w})\Lambda_\zeta(w)\Lambda_\zeta(w+1)\bigr\}.
\end{aligned}
\end{equation}
Replacing $w$ by $-w$ interchanges the first two products and reverses
the sign of the last term, making the right side odd. This recovers
\eqref{log:reflection}. Together with the pole statement in
Proposition~\ref{log:transform}, it shows that $w2^{-w}G(w/2)$ is even
and entire, and gives
\begin{equation}\label{restore:laurent}
 G(s)=\frac{h_0}{s}+h_0\log4+O(s),\qquad
 G(-\tfrac12)=-\tfrac14G(\tfrac12).
\end{equation}
The negative value is understood by meromorphic continuation.

Taking logarithms of the product defining $F$ gives
\begin{equation}\label{restore:logF}
 \int_0^\infty x^{s-1}\log F(e^{-\pi\sqrt{x}})\,dx
   =-2\Gamma(2s)\pi^{-2s}\zeta(2s+1)L(2s,\chi)
   \qquad(\Re s>0).
\end{equation}
Indeed, expand $\log F(q)=-\sum_{j,k\ge1}\chi(j)q^{jk}/k$ and
integrate absolutely for $\Re s>1/2$. Both sides are holomorphic for
$\Re s>0$; for the integral this follows from
$\log F(q)\to-\log\phi$ as $q\to1^-$ and exponential decay at the
other endpoint. The identity theorem proves \eqref{restore:logF}
on the stated half-plane. At integer orders the even-value formula
for $L$ reduces it to
\begin{equation}\label{restore:logF-integers}
 \int_0^\infty x^{n-1}\log F(e^{-\pi\sqrt{x}})\,dx
 =\frac{(-1)^n\sqrt5}{2n}\left(\frac25\right)^{2n}
       B_{2n,\chi}\zeta(2n+1),\qquad n\ge1.
\end{equation}
For $n=1$ this equals $-8\zeta(3)/(25\sqrt5)$.

\section*{Statements and declarations}
\noindent\textbf{Funding.} The author received no funding or financial support for this work.

\smallskip\noindent\textbf{Competing interests.} The author declares no competing interests.

\smallskip\noindent\textbf{Data availability.} This article is theoretical. Its results follow from the formulas and proofs given in the text; no external dataset is required.

\end{document}